\documentclass[
final,
nomarks
]{dmtcs-episciences}

\usepackage{amsmath}
\usepackage{amssymb}
\usepackage{thmtools}
\usepackage{mathtools}

\DeclareMathOperator{\scol}{scol}
\DeclareMathOperator{\wcol}{wcol}

\DeclareMathOperator{\tw}{tw}
\DeclareMathOperator{\cw}{copwidth}
\DeclareMathOperator{\fw}{flipwidth}

\usepackage[capitalise, compress, nameinlink, noabbrev]{cleveref}
\crefname{lem}{Lemma}{Lemmas}
\crefname{thm}{Theorem}{Theorems}
\crefname{ques}{Question}{Theorems}
\crefname{cor}{Corollary}{Corollaries}
\crefname{enumi}{Item}{Items}
\crefformat{equation}{(#2#1#3)}
\Crefformat{equation}{Equation #2(#1)#3}
\crefformat{enumi}{#2#1#3}
\Crefformat{enumi}{Item (#2#1#3)}

\newcommand{\NN}{\mathbb{N}}

\newcommand{\PP}{\mathcal{P}}

\newcommand{\GG}{\mathcal{G}}

\newcommand{\WW}{\mathcal{W}}

\DeclarePairedDelimiter{\abs}{\lvert}{\rvert}
\DeclarePairedDelimiter{\set}{\{}{\}} 

\newtheorem{thm}{Theorem}
\newtheorem{lem}[thm]{Lemma}
\newtheorem{cor}[thm]{Corollary}

\usepackage[utf8]{inputenc}
\usepackage{subfigure}

\usepackage[round]{natbib}

\author[Robert Hickingbotham]{Robert Hickingbotham\affiliationmark{1}}
\title[Cop-Width, Flip-Width and Strong Colouring Numbers]{Cop-Width, Flip-Width and Strong Colouring Numbers}
\affiliation{
	CNRS, ENS de Lyon, Université Claude Bernard Lyon 1, LIP, UMR 5668,
	Lyon, France}
\begin{document}
\publicationdata{vol. 27:2}{2025}{14}{10.46298/dmtcs.14976}{2024-12-20; 2024-12-20; 2025-04-17}{2025-05-04}
\maketitle

\begin{abstract}
  \vspace{2ex} 
  Cop-width and flip-width are new families of graph parameters introduced by Toru{\'n}czyk (2023) that generalise treewidth, degeneracy, generalised colouring numbers, clique-width and twin-width. In this paper, we bound the cop-width and flip-width of a graph by its strong colouring numbers. In particular, we show that for every $r\in \NN$, every graph $G$ has $\cw_r(G)\leq \scol_{4r}(G)$. This implies that every class of graphs with linear strong colouring numbers has linear cop-width and linear flip-width. We use this result to deduce improved bounds for cop-width and flip-width for various sparse graph classes.  
\end{abstract}

\section{Introduction}    
\cite{KY2003orderings} introduced the following definitions\footnote{We consider simple, finite, undirected graphs $G$ with vertex-set ${V(G)}$ and edge-set ${E(G)}$. See \citep{diestel2017graphtheory} for graph-theoretic definitions not given here. Let ${\NN := \{1,2,\dots\}}$. For integers $a,b$ where $a\leq b$, let $[a,b]:=\{a,a+1,\dots,b-1,b\}$ and for $n \in \NN$, let $[n]:=[1,n]$. A \emph{graph class} is a collection of graphs closed under isomorphism. For a graph $G$ and a vertex $v\in V(G)$, let ${N_G(v)}:=\{w\in V(G)\colon vw\in E(G)\}$.}. For a graph $G$, a total order $\preceq$ of $V(G)$, a vertex $v\in V(G)$, and an integer $r\geq 1$, let ${R(G,\preceq,v,r)}$ be the set of vertices $w\in V(G)$ for which there is a path $(v=w_0,w_1,\dots,w_{r'}=w)$ of length $r'\in [0,r]$ such that $w\preceq v$ and $v\prec w_i$ for all $i\in [r'-1]$, and let ${Q(G,\preceq,v,r)}$ be the set of vertices $w\in V(G)$ for which there is a path $v=w_0,w_1,\dots,w_{r'}=w$ of length $r'\in [0,r]$ such that $w\preceq v$ and $w\prec w_i$ for all $i\in [r'-1]$. For a graph $G$ and integer $r\geq 1$, the \emph{$r$-strong colouring number of $G$, $\scol_r(G)$,} is the minimum integer such that there is a total order~$\preceq$ of $V(G)$ with $|R(G,\preceq,v,r)|\leq \scol_r(G)$ for every vertex $v$ of $G$. Likewise, the \emph{$r$-weak colouring number of $G$, $\wcol_r(G)$,} is the minimum integer such that there is a total order~$\preceq$ of $V(G)$ with $|Q(G,\preceq,v,r)|\leq \wcol_r(G)$ for every vertex $v$ of $G$. 

Generalised colouring numbers provide upper bounds on several graph parameters of interest. First note that $\scol_1(G)=\wcol_1(G)$ which equals the degeneracy of $G$ plus 1, implying $\chi(G)\leq \scol_1(G)$. A proper graph colouring is \emph{acyclic} if the union of any two colour classes induces a forest; that is, every cycle is assigned at least three colours. For a graph $G$, the \emph{acyclic chromatic number of $G$, $\chi_\text{a}(G)$,} is the minimum integer $k$ such that $G$ has an acyclic $k$-colouring.
\citet{KY2003orderings} proved that $\chi_\text{a}(G)\leq \scol_2(G)$ for every graph $G$. Other parameters that can be bounded by strong and weak colouring numbers include game chromatic number \citep{KT1994uncooperative,KY2003orderings}, Ramsey numbers \citep{CS1993ramsey}, oriented chromatic number \citep{KSZ1997acyclic}, arrangeability~\citep{CS1993ramsey}, boxicity \citep{EW2018boxicity}, odd chromatic number \citep{H2022odd} and conflict-free chromatic number \citep{H2022odd}.

Another attractive aspect of strong colouring numbers is that they interpolate between degeneracy and treewidth\footnote{
	A \emph{tree-decomposition} of a graph $G$ is a collection $\WW = (B_x \colon x \in V(T))$ of subsets of $V(G)$ indexed by the nodes of a tree $T$ such that
	(i) for every edge $vw \in E(G)$, there exists a node $x \in V(T)$ with $v,w \in B_x$; and 
	(ii) for every vertex $v \in V(G)$, the set $\set{x \in V(T) \colon v \in B_x}$ induces a (connected) subtree of $T$.  
	The \emph{width} of $\WW$ is $\max\set{\abs{B_x} \colon x \in V(T)} - 1$. 
	The \emph{treewidth $\tw(G)$} of a graph $G$ is the minimum width of a tree-decomposition of $G$. }. 
As previously noted, $\scol_1(G)$ equals the degeneracy of $G$ plus $1$. At the other extreme, \citet{GKRSS2018coveringsnowhere} showed that $\scol_r(G)\leq \tw(G)+1$ for every $r\in \NN$, and indeed 
$\scol_r(G) \to \tw(G)+1$ as $r\to\infty$.

Generalised colouring numbers are important because they characterise bounded expansion \citep{zhu2009generalized} and nowhere dense classes \citep{GKRSS2018coveringsnowhere}, and have several algorithmic applications \cite{dvorak2014approximation,GKS2017propertiesnowhere}.  Let $G$ be a graph and $r\geq 0$ be an integer. A graph $H$ is an \emph{$r$-shallow minor} of $G$ if $H$ can be obtained from a subgraph of $G$ by contracting disjoint subgraphs each with radius at most $r$. Let \emph{$G \,\triangledown \, r$} be the set of all {$r$-shallow-minors} of $G$, and let 
\emph{$\nabla_r(G)$}$:=\max\{|E(H)|/|V(H)|\colon H\in G \,\triangledown \, r \}$.
A hereditary graph class $\GG$ has \emph{bounded expansion} with \emph{expansion function} $f_{\GG}: \NN\cup \{0\} \to \mathbb{R}$ if $\nabla_r(G)\leq f_{\GG}(r)$ for every $r \geq 0$ and graph $G \in \GG$.
Bounded expansion is a robust measure of sparsity with many characterisations \cite{zhu2009generalized,nevsetvril2008grad,nevsetvril2012sparsity}. For example, \citet{zhu2009generalized} showed that  a hereditary graph class $\GG$ has bounded expansion if there is a function $f$ such that $\scol_r(G)\leq f(r)$ for every $r \geq 1$ and graph $G \in \GG$.
Examples of graph classes with bounded expansion include classes that have bounded maximum degree \cite{nevsetvril2008grad}, bounded stack number \cite{NOW2012examples}, bounded queue-number \cite{NOW2012examples}, bounded nonrepetitive chromatic number \cite{NOW2012examples}, strongly sublinear separators \cite{DN2016sublinear}, as well as proper minor-closed graph classes \cite{nevsetvril2008grad}. See the book by \citet{nevsetvril2012sparsity} for further background on bounded expansion.

Given the richness of generalised colouring numbers, several attempts have been made to extend these parameters to the dense setting. In a recent breakthrough, \citet{torunczyk2023flip} introduced cop-width and flip-width, new families of graph parameters that generalise treewidth, degeneracy, generalised colouring numbers, clique-width and twin-width. Their definitions are inspired by a game of cops and robber by \citet{seymour1993graph}:

\textit{“The robber stands on a vertex of the graph, and can at any time run at great speed to any other vertex along a path of the graph. He is not permitted to run through a cop, however. There are $k$ cops, each of whom at any time either stands on a vertex or is in a helicopter (that is, is temporarily removed from the game). The objective of the player controlling the movement of the cops is to land a cop via helicopters on the vertex occupied by the robber, and the robber’s objective is to elude capture. (The point of the helicopters is that cops are not constrained to move along paths of the graph – they move from vertex to vertex arbitrarily.) The robber can see the helicopter approaching its landing spot and may run to a new vertex before the helicopter actually lands”} \cite{seymour1993graph}.

\citet{seymour1993graph} showed that the least number of cops needed to win this game on a graph $G$ is in fact equal to $\tw(G)+1$, thus giving a min-max theorem for treewidth. \citet{torunczyk2023flip} introduced the following parameterised version of this game: for some fixed $r\in \NN$, the robber runs at speed $r$. So in each round, after the cops have taken off in their helicopters to their new positions (they may also choose to stay put), which are known to the robber, and before the helicopters have landed, the robber may traverse a path of length at most $r$ that does not run through a cop that remains on the ground. This variant is called the \emph{cop-width game with radius $r$ and width $k$}, if there are $k$ cops, and the robber can run at speed $r$. For a graph $G$, the \emph{radius-$r$ cop-width of $G$, $\cw_r(G)$,} is the least number $k\in \NN$ such that the cops have a winning strategy for the cop-width game played on $G$ with radius $r$ and width $k$.

Say a class of graphs $\GG$ has \emph{bounded cop-width} if there is a function $f$ such that for every $r\in \NN$ and graph $G\in \GG$, $\cw_r(G)\leq f(r)$. \citet{torunczyk2023flip} showed that bounded cop-width coincides with bounded expansion.

\begin{thm}[\cite{torunczyk2023flip}]
	A class of graphs has bounded expansion if and only if it has bounded cop-width.
\end{thm}

As such, only sparse graph classes have bounded cop-width. Flip-width is then defined as a dense analog of cop-width. Here, the cops have enhanced power where they are allowed to perform flips on subsets of the vertex set of the graph with the goal of isolating the robber. For a fixed graph $G$, applying a flip between a pair of sets of vertices $A, B \subseteq V(G)$ results in the graph obtained from $G$ by inverting the adjacency between any pair of vertices $a, b$ with $a \in A$ and $b \in B$. If $G$ is a graph and $\PP$ is a partition of $V(G)$, then call a graph $G'$ a \emph{$\PP$-flip} of $G$ if $G'$ can be obtained from $G$ by performing a sequence of flips between pairs of parts $A, B \in \PP$ (possibly with $A = B$). Finally, call $G'$ a \emph{$k$-flip} of $G$, if $G'$ is a $\PP$-flip of $G$, for some partition $\PP$ of $V(G)$ with $|\PP|\leq k$.

The \emph{flip-width game with radius $r\in \NN$ and width $k \in \NN$} is played on a graph $G$ by two players, flipper and runner. Initially, $G_0 = G$ and $x_0$ is a vertex of $G$ chosen by the runner. In each round $i\geq 1$ of the game, the flipper announce a new $k$-flip $G_i$ of $G$. The runner, knowing $G_i$, moves to a new vertex $x_i$ by running along a path of length at most $r$ from $x_{i-1}$ to $x_i$ in the previous graph $G_{i-1}$. The game terminates when $x_i$ is an isolated vertex in $G_i$. For a fixed $r \in \NN$, the \emph{radius-$r$ flip-width of a graph $G$, $\fw_r(G)$}, is the least number $k \in \NN$ such that the flipper has a winning strategy in the flip-width game of radius $r$ and width $k$ on $G$.

In contrast to cop-width, flip-width is well-behaved on dense graphs. For example, one can easily observe that for all $r\in \NN$, the radius-$r$ flip-width of a complete graph is equal to 1. Moreover, to demonstrate the robustness of flip-width, \citet{torunczyk2023flip} proved the following results.

\begin{thm}[\cite{torunczyk2023flip}] \,
	\begin{itemize}
		\item Every class of graphs with bounded expansion has bounded flip-width.
		\item Every class of graphs with bounded twin-width has bounded flip-width.
		\item If a class of graphs $\GG$ has bounded flip-width, then any first-order interpretation of $\GG$ also has bounded flip-width.
		\item There is a slicewise polynomial algorithm that approximates the flip-width of a given graph $G$.
	\end{itemize}
\end{thm}
So flip-width is considered to be a good analog of generalised colouring numbers for dense graphs. See \cite{torunczyk2023flip,CKKL2023flipwidth,EM2023geometric} for further results and conjectures on flip-width.

\subsection*{Results}
In this paper, we bound the cop-width of a graph by its strong colouring numbers.

\begin{restatable}{thm}{MainStrongColouring}\label{MainStrongColouring}
	For every $r\in \NN$, every graph $G$ has $\cw_r(G)\leq \scol_{4r}(G)$.
\end{restatable}

Previously, the best known bounds for the cop-width of a sparse graph was through its weak-colouring numbers. \citet{torunczyk2023flip} showed that for every $r\in \NN$, every graph $G$ has 
$$\cw_r(G)\leq \wcol_{2r}(G)+1.$$ 
Moreover, if $G$ excludes $K_{t,t}$ as a subgraph, then $\fw_r(G)\leq (\cw_r(G))^t.$ While graph classes with bounded strong colouring numbers have bounded weak colouring numbers, strong colouring numbers often give much better bounds than weak colouring numbers. In fact, \citet{GKRSS2018coveringsnowhere} and \citet{DPTY2022weak} have both shown that there are graph classes with polynomial strong colouring numbers and exponential weak colouring numbers.


\begin{restatable}{thm}{LinearStrongCopFlip}\label{LinearStrongCopFlip}
	Every class of graphs with linear cop-width has linear flip-width.
\end{restatable}
It therefore follows that that every graph class with linear strong colouring numbers has linear cop-width and linear flip-width.

Second, \cref{MainStrongColouring} gives improved bounds for the cop-width of many well-studied sparse graphs. A graph $H$ is a \emph{minor} of a graph $G$ if $H$ is isomorphic to a graph that can be obtained from a subgraph of $G$ by contracting edges. A graph $G$ is \emph{$H$-minor-free} if $H$ is not a minor of $G$. Van den Heuvel, Ossona de Mendez, Quiroz, Rabinovich and Siebertz~\cite{HMQRS2017fixed} showed that for every $r\in \NN$, every $K_t$-minor-free graph $G$ has $\scol_r(G)\leq \binom{t-1}{2}(2r+1)$. So \cref{MainStrongColouring} implies the following.

\begin{restatable}{thm}{KtMain}\label{KtMain}
	For all $r,t\in \NN$, every $K_t$-minor-free graph $G$ has
	$$\cw_r(G)\leq \binom{t-1}{2}(8r+1).$$
\end{restatable}

By \cref{LinearStrongCopFlip}, $K_t$-minor-free graphs also have linear flip-width; see \cref{KtMainFlipWidth} for an explicit bound. In regard to the previous best known bound for this class of graphs, van~den~Heuvel~et~al.~\cite{HMQRS2017fixed} showed that for every $r\in \NN$, every $K_t$-minor-free graph $G$ has $\wcol_r(G)\in O_t(r^t)$. By the aforementioned result of \citet{torunczyk2023flip}, the previous best known bound for the cop-width and flip-width of $K_t$-minor-free graph $G$ was:
\begin{equation*}
	\cw_r(G) \in O_t(r^{t-1})\text{ and } \fw_r(G) \in O_t(r^{(t-1)^2}).
\end{equation*}

\cref{MainStrongColouring} is also applicable to non-minor-closed graph classes. For a surface $\Sigma$, we say that a graph $G$ is \emph{$(\Sigma, k)$-planar} if $G$ has a drawing on $\Sigma$ such that every edge of $G$ is involved in at most $k$ crossings. A graph is \emph{$(g,k)$-planar} if it is $(\Sigma, k)$-planar for some surface $\Sigma$ with Euler genus at most $g$. Such graphs are widely studied \citep{PT1997crossings,DMW20,DEW2017locally,dujmovic2017layered} and are a classic example of a sparse non-minor-closed class of graphs.

Van den Heuvel and Wood~\cite{HW2018improperARXIV,HW2018improper} showed that for every $r\in \NN$, every $(g,k)$-planar graph $G$ has $\scol_r(G)\leq (4g+6)(k+1)(2r+1)$. So \cref{MainStrongColouring} implies the following.

\begin{thm}\label{cwgkplanar}
	For all $g,k,r\in \NN$, every $(g,k)$-planar graph $G$ has $$\cw_r(G)\leq (4g+6)(k+1)(8r+1).$$
\end{thm}

See \cite{HW2021shallow,HW2018improperARXIV,HMQRS2017fixed,DMN2021convex} for other graph classes that \cref{MainStrongColouring} applies to.

\section{Proofs}
We now prove our main theorem.

\MainStrongColouring*

\begin{proof}
	Let $n:=|V(G)|-1$ and let $(v_0,v_1,\dots,v_n)$ be a total order $\preceq$ of $V(G)$ where $|R(G,\preceq,v,4r)|\leq \scol_{4r}(G)$ for every $v\in V(G)$. For every $s\in \NN$ and $v_i,v_j\in V(G)$ where $i\leq j$, let ${M(v_i,v_j,s)}$ be the set of vertices $w\in V(G)$ for which there is a path $v_j=w_0,w_1,\dots,w_{s'}=w$ of length $s'\in [0,s]$ such that $w\preceq v_i$ and $v_i\prec w_{\ell}$ for all $\ell\in [s'-1]$. 
	
	\textbf{Claim:} For all $v_i,v_j\in V(G)$ where $i\leq j$, $|M(v_i,v_j,2r)|\leq \scol_{4r}(G)$.
	
	\begin{proof}
		Let $k\in [i,j]$ be minimal such that $v_k\in Q(G,\preceq,v_j,2r)$. So $G$ contains a path $P=(v_j=u_0,\dots,u_{r'}=v_k)$ of length $r'\in [0,2r]$ such that $v_k\prec u_{\ell}$ for all $\ell\in [r'-1]$. We claim that $M(v_i,v_j,2r)\subseteq R(G,\preceq,v_k,4r)$. Let $w \in M(v_i,v_j,2r)$. Then there is a path $P'=(v_j=w_0,\dots,w_{s'}=w)$ of length $s'\in [0,2r]$ such that $w \preceq v_{i}$ and $v_i\prec w_{\ell}$ for all $\ell \in [s'-1]$. Suppose there is an $\ell \in [s'-1]$ such that $w_{\ell} \prec v_k$. Choose $\ell$ to be minimum. Then $w_{\ell}\in Q(G,\preceq,v_j,2r)$ since $w_{\ell}\prec v_k \preceq w_a$ for each $a\in [0,\ell-1]$, which contradicts the choice of $k$. So $v_k\preceq w_{\ell}$ for all $\ell \in [s'-1]$. By taking the union of $P$ and $P'$, it follows that $G$ contains a $(v_k,w)$-walk $W$ of length at most $4r$ such that $v_k \prec u$ for all $u\in V(W)\setminus\{v_k,w\}$. Therefore $w\in R(G,\preceq,v_k,4r)$ and so $|M(v_i,v_j,2r)|\leq |R(G,\preceq,v_k,4r)|\leq \scol_{4r}(G).$
	\end{proof}
	
	For each round $i\geq 0$ until the robber is caught, we will define a tuple $(v_i,x_i,C_i,D_i,V_i,P_i)$ where: 
	
	\begin{enumerate}[(1)]
		\item\label{A1} $x_i\in V(G)$ is the location of the robber at the end of round $i$;
		\item\label{A4} $C_0:=\{v_0\}$ and $C_i:=M(v_i,x_{i-1},2r)$ is the set of vertices that the cops are on at the end of round $i$ for each $i\geq 1$;
		\item\label{A5} $D_0:=\emptyset$ and $D_i:=C_{i-1}\cap C_i$ is the set of vertices where cops remain put throughout round $i$ for each $i\geq 1$; 
		\item\label{A7} $V_i:=\{v_0,\dots,v_i\}$ where $v_i$ is defined by the total order $\preceq$ of $V(G)$; and
		\item\label{A6} $P_0:=\emptyset$ and $P_i$ is the $(x_{i-1},x_i)$-path of length at most $r$ that the robber traverses during round $i$ for each $i\geq 1$.
	\end{enumerate}
	
	Note that we require $v_i\preceq x_{i-1}$ in order for $M(v_i,x_{i-1},2r)$ to be well-defined. We will prove that the tuple is indeed well-defined by inductively maintain the following invariants for each round $i\geq 0$: 
	
	\begin{enumerate}[(1)]
		\setcounter{enumi}{5}
		\item\label{Inv1} $v_i \preceq x_i$;
		\item\label{Inv2} if $i\geq 1$, every path in $G$ from $x_{i-1}$ to a vertex in $V_{i-1}$ of length at most $r$ contains a vertex from $D_{i}$;
		\item\label{Inv3} $M(v_i,x_{i},r)\subseteq C_i$;
		\item\label{Inv4} if $i\geq 1$, then $V(P_i)\cap V_{i-1}=\emptyset$; and
		\item\label{Inv5} if $v_i=x_i$, then the robber is caught.
	\end{enumerate}
	
	
	
	
	Together with the previous claim, \labelcref{A4,Inv1,Inv5} imply that the robber is caught within $n$ rounds using at most $\scol_{4r}(G)$ cops.
	
	Initialise the game of cops and robber with the robber on some vertex $x_0$ in $V(G)$, one cop on $v_0$ and the remaining cops all in the helicopters. Define the tuple $(v_0,x_0,C_0,D_0,V_0,P_0)$ according to \labelcref{A1,A4,A5,A7,A6}. Clearly such a tuple is well-defined. Moreover, it is easy to see that the tuple satisfies \labelcref{Inv1,Inv2,Inv3,Inv4,Inv5}.  
	
	Now suppose we are at round $i \geq 1$ and the robber has not yet been caught. By induction, we may assume that there is a tuple $(v_{i-1},x_{i-1},C_{i-1},D_{i-1},V_{i-1},P_{i-1})$ for round $i-1$ which satisfies \labelcref{A1,A4,A5,A7,A6,Inv1,Inv2,Inv3,Inv4,Inv5}. Since the robber has not yet been caught, \labelcref{Inv1,Inv5} imply that $v_{i-1}\prec x_{i-1}$, so $v_i \preceq x_{i-1}$. Therefore, there is a well-defined tuple $(v_i,x_i,C_i,D_i,V_i,P_i)$ which satisfies \labelcref{A1,A4,A5,A7,A6}. We now show that $(v_i,x_i,C_i,D_i,V_i,P_i)$ satisfies the additional invariants.
	
	We first verify \labelcref{Inv2}. Let $F_{i}:=M(v_{i-1},x_{i-1},r)$. Let $u\in V_{i-1}$ and suppose there is a  path $P^{\star}= (x_{i-1}=w_0,w_1,\dots,w_{r'}=u)$ in $G$ where $r'\in [0,r]$. Consider the minimal $j\in [r']$ such that $w_j\in V_{i-1}$. Since $\{w_1,\dots,w_{j-1}\}\cap V_{i-1}=\emptyset$, it follows that $w_j\in F_i$. So for every $u\in V_{i-1}$, every $(x_{i-1},u)$-path in $G$ of length at most $r$ contains a vertex from $F_i$. By \labelcref{A4} and \labelcref{Inv3} (from the $i-1$ case), it follows that $F_i\subseteq C_{i-1}\cap C_i$. So \labelcref{Inv2} follows from \labelcref{A5}. Now since the robber is not allowed to run through a cop that stays put, \labelcref{Inv4} follows by \labelcref{A5,A6,Inv2}. Property~\labelcref{Inv1} then immediately follows from \labelcref{Inv4} since $x_i\in V(P_i)$. Now consider a vertex $y\in M(v_i,x_i,r)$. Then $G$ contains a path $P'=(x_i=w_0,w_1,\dots,w_{r'}=y)$ of length $r'\in [0,r]$ such that $v_i\preceq w_j$ for all $j\in [r'-1]$. By taking the union of $P'$ and $P_i$, it follows that $G$ contains an $(x_{i-1},y)$-walk $W$ of length at most $2r$. Moreover, by \labelcref{Inv2} and the definition of $P'$, $v_i\preceq z$ for all $z\in V(W)\setminus\{x_{i-1},y\}$. So $y \in M(v_i,x_{i-1},2r)$ and thus \labelcref{A4} implies \labelcref{Inv3}. Finally, if $v_i=x_i$ then \labelcref{Inv3} implies $x_i\in C_i$, so \labelcref{Inv5} follows by \labelcref{A1,A4}, as required.
\end{proof}

To prove \cref{LinearStrongCopFlip}, we leverage known results concerning neighbourhood diversity. Neighbourhood diversity is a well-studied concept with various applications \cite{RVS19,EGKKPRS17,GHOORRVS17,PP20,BKW2022bandwidth,BFLP2024neighbourhood,JR2023neighborhood}. Let $G$ be a graph. For a set $S\subseteq V(G)$, let ${\pi_G(S)}:=|\{N_G(v)\cap S \colon v\in V(G)\setminus S\}|$. For $k\in \NN$, let ${{\pi_G(k)}:=\max\{\pi_G(S)\colon S\subseteq V(G), |S|\leq k\}}$.

\begin{lem}\label{FlipWidthLinear}
	For all $k,r\in \NN$, every graph $G$ with $\cw_r(G)\leq k$ has 
	$${\fw_r(G)\leq \pi_G(k)+k}.$$
\end{lem}

\begin{proof}
	We claim that for every set $S\subseteq V(G)$ where $|S|\leq k$, there is a $(\pi_G(k)+k)$-flip that isolates $S$ while leaving $G-S$ untouched. Let $\mathcal{P}$ be a partition of $V(G)$ that partition $S$ into singleton and vertices in $v\in V(G)\setminus S$ according to $N_G(v)\cap S$. Then $|\mathcal{P}|\leq \pi_G(k)+k$. Moreover, every vertex $s\in S$ can be isolated by flipping $\{s\}$ with every part of $\mathcal{P}$ that is complete to $\{s\}$. Thus, a winning strategy for the cops in the cop-width game with radius $r$ and width $k$ can be transformed into a winning strategy for the flipper in the flip-width graph with radius $r$ and width $\pi_G(k)+k$, as required. 
\end{proof}

\citet{RVS19} showed that for every graph class $\GG$ with bounded expansion, there exists $c>0$ such that $\pi_G(k)\leq ck$ for every $G\in \mathcal{G}$. Since graph classes with linear cop-width have bounded expansion, \cref{FlipWidthLinear} imply \cref{LinearStrongCopFlip}. As a concrete example, \citet{BKW2022bandwidth} showed that for every $K_t$-minor-free graph $G$ and for every set $A\subseteq V(G)$, 
$$\pi_G(A)\leq 3^{2t/3+o(t)}|A|+1.$$
So \cref{KtMain,FlipWidthLinear} imply the following.

\begin{cor}\label{KtMainFlipWidth}
	For all $r,t\in \NN$, every $K_t$-minor-free graph $G$ has
	$$\fw_r(G)\leq 3^{(2t/3+o(t))}t^2r.$$
\end{cor}

\subsubsection*{Acknowledgement} This work was initiated at the 10th Annual Workshop on Geometry and Graphs held at Bellairs Research Institute in February 2023. Thanks to Rose McCarty for introducing the author to flip-width and to David Wood for helpful conversations. The author was supported by an Australian Government Research Training Program Scholarship when he completed this work.
%
\bibliographystyle{DavidNatbibStyle}
\bibliography{main-journal.bbl}
\label{sec:biblio}


\end{document}